\documentclass[11pt]{amsart}

\usepackage{amsthm,amssymb,amsfonts,hyperref,hhline,amsmath,mathtools,enumerate,graphicx,tikz,listings,array,multicol,pgfplots,colortbl,makecell,xfrac}

\usepackage[margin=1in]{geometry}

\usetikzlibrary{shapes.geometric, arrows}
\tikzstyle{weights} = [rectangle, rounded corners, minimum width=1cm, minimum height=1cm,text centered, draw=black]
\tikzstyle{biases} = [rectangle, rounded corners, minimum width=1cm, minimum height=1cm,text centered, draw=black]
\tikzstyle{Relu} = [rectangle, rounded corners, minimum width=1cm, minimum height=1cm,text centered, draw=black, fill=yellow!30]
\tikzstyle{Li} = [rectangle, dashed, rounded corners, minimum width=1cm, minimum height=1cm,text centered, draw=black]
\tikzstyle{li} = [rectangle, dashed, rounded corners, minimum width=.5cm, minimum height=.5cm,text centered, draw=black]
\tikzstyle{lweights} = [rectangle, rounded corners, minimum width=.5cm, minimum height=.5cm,text centered, draw=black, fill=gray!30]
\tikzstyle{lbiases} = [rectangle, rounded corners, minimum width=.5cm, minimum height=.5cm,text centered, draw=black, fill=gray!30]

\DeclareMathOperator{\cone}{cone}
\DeclareMathOperator{\conv}{conv}

\DeclareMathOperator{\BCE}{BCE}
\DeclareMathOperator{\HB}{\mathbb{HB}}

\DeclareMathOperator{\IDP}{\mathbb{IDP}}
\DeclareMathOperator{\height}{height}
\DeclareMathOperator{\supp}{\bf{supp}}
\DeclareMathOperator{\cutoff}{\bf{cutoff}}

\newcommand{\R}{\mathbb{R}}

\newcommand{\cS}{\mathcal{S}}
\newcommand{\Z}{\mathbb{Z}}

\renewcommand{\phi}{\varphi}

\def\x{{\boldsymbol x}}

\newcommand\commentout[1]{}

\newtheorem{theorem}{Theorem}[section]

\newtheorem{corollary}[theorem]{Corollary}
\newtheorem{proposition}[theorem]{Proposition}

\theoremstyle{remark}
\newtheorem{example}[theorem]{Example}
\newtheorem{remark}[theorem]{Remark}
\theoremstyle{definition}
\newtheorem{definition}[theorem]{Definition}

\begin{document}

\title{Predicting the Integer Decompostion Property\\ via Machine Learning}

\author{Brian Davis}
\address{Department of Mathematics\\
         University of Kentucky\\
         Lexington, KY 40506--0027}
\email{brian.davis@uky.edu}

\date{\today}

\begin{abstract}
In this paper we investigate the ability of a neural network to approximate algebraic properties associated to lattice simplices. In particular we attempt to predict the distribution of Hilbert basis elements in the fundamental parallelepiped, from which we detect the integer decomposition property (IDP). We give a gentle introduction to neural networks and discuss the results of this prediction method when scanning very large test sets for examples of IDP simplices. 
\end{abstract}

\maketitle
\section{Introduction} 
Due to the maturity and ubiquity of machine learning techniques and applications, open--source software libraries such as {\bf{Tensorflow}} have become available to non-specialists. These libraries are typically well--documented, and friendly technical references are freely available online, e.g., \cite{DeepLearningBook}. In this environment, it seems natural to ask:
\begin{center}\emph{How do we apply machine learning technology to algebraic combinatorics?}
\end{center}
It is not clear how to extract human--understandable meaning from the raw numerical data of, for example, a neural network (for a discussion of comprehensibility, see \cite{understandability}.) We therefore  employ these techniques for their prediction and approximation power, rather than for use in theorems and their proofs. There is a long history of using neural networks in order to approximate solutions to combinatorial optimization problems, e.g. the traveling salesman problem \cite{travelingSalesman}, and in \cite{conjugacy}, Gryak, Haralick, and Kahrobaei use machine learning to predict if two elements of a group are conjugate. It seems reasonable, then, to hope that machine learning has some applicability to problems at the intersection of combinatorics and algebra.

We intend for this work to be an introduction to neural networks and a proof of concept for the use of machine learning, and neural networks in particular, in predicting properties relevant to lattice points in polyhedra. As a particular application, we attempt to predict the integer decomposition property (IDP) in a special class of lattice simplices.

In their paper \cite{2016arXiv160801614B}, Braun, Davis, and Solus study the infinite family of lattice simplices of the form \[
\Delta_{(1,q)}=\conv\left\{e_1,\dots,e_d,-\sum_{i=1}^dq_ie_i\right\}\subset\R^d,
\]
where $q_i\in\Z_{\geq0}$ for all $i$, and give sufficient conditions on the entries $q_i$ (the {\bf{$q$-vector}}) for a such a simplex to be IDP in the case that it is reflexive. In the present work we will ``train" a neural network to predict if a given example of a $\Delta_{(1,q)}$ simplex is IDP without actually computing the Hilbert basis.

In Section 2, after presenting the basics of Hilbert bases, we interpret the integer decomposition property as a composition of functions to be approximated. In Section 3 we develop the general framework for training a neural network using the language of piece-wise linear functions\footnote{This exposition agrees with the more common descriptions of neural networks when restricted to the case that the source of the training data is a well-defined function and we use {\bf{ReLU}} activation functions.} and stochastic gradient descent. In Section 4 we discuss a piece-wise linear approximation of the integer decomposition property and its accuracy. 
 
\section{Approximating the integer decomposition property as a function}
We will now introduce the Hilbert basis and use it to define the integer decomposition property. We encode the integer decomposition property as a real valued function $\IDP$ in Subsection \ref{IDPdefinition}, and define what it means to approximate the integer decomposition property.
\subsection{Hilbert Basics}
Let $v_1,\dots,v_{d+1}$ be affinely independent elements of the integer lattice $\Z^d$. Their convex hull is a $d$-simplex 
\[\Delta := \left\{\sum_{i=1}^{d+1}\gamma_iv_i\;:\;0\leq \gamma_i\;, \; \sum_{i=1}^{d+1}\gamma_i=1\right\}\subset\R^d
,\] 
and we define $\cone(\Delta)$ to be the non-negative real span of the points in $(1,\Delta)$, i.e., $\Delta$ embedded into $\R^{d+1}$ at height 1 in the zeroth coordinate:
\[\cone(\Delta) := \left\{\sum_{i=1}^{d+1}\gamma_i(1,v_1)\;:\; 0\leq\gamma_i\right\}\subset\R^{d+1}
.\]
The set $\cone(\Delta)\cap\Z^{d+1}$ is closed under addition, and the unique minimal collection of additive generators is called 
the {\bf{Hilbert basis}} of $\cone(\Delta)$.

We call the zeroth coordinate of a point $z$ in $\cone(\Delta)\cap\Z^{d+1}$ the {\bf{height}} of $z$, denoted \[\height(z):=z_0,\]
and we say that the simplex $\Delta$ has the {\bf{integer decomposition property}} (IDP) if, for each element of the Hilbert basis of $\cone(\Delta)$, the height is equal to 1.

Define the {\bf{fundamental parallelepiped}} of $\Delta$ to be the weighted sum of the cone generators $(1,v_i)$ with non-negative weights strictly less than 1:
\[
\Pi_\Delta:=\left\{\sum_{i=1}^{d+1}\gamma_i(1,v_i) \, : \, 0\leq\gamma_i<1\right\}\subset\cone(\Delta) \, .
\]
Because any element $z$ of $\cone(\Delta)\cap\Z^{d+1}$ lies in $\cone(\Delta)$, it is a non-negative linear combination of the $(1,v_i)$'s, i.e., there exist non-negative real coefficients $g_i$ such that \[z=\sum_{i=1}^{d+1}g_i(1,v_i)=\left(\sum_{i=1}^{d+1}\left\lfloor g_i\right\rfloor (1,v_i)\right) + \left(\sum_{i=1}^{d+1}\{g_i\} (1,v_i)\right)\] where $\{g_i\}$ means the fractional part of $g_i$. By setting $\gamma_i$ equal to $\{g_i\}$, we see that any point $z$ may be written  as a non-negative {\emph{integral}} combination of the $(1,v_i)$'s and an integer point in $\Pi_\Delta\cap\Z^{d+1}$. Consequently, the Hilbert basis consists of the cone generators $(1,v_i)$ together with the additively minimal elements of $\Pi_\Delta\cap\Z^{d+1}$.

\subsection{Partitioning $\Pi$} We partition $\Pi_\Delta$ into disjoint subsets we call {\bf{bins}} $B_\alpha$ for $\alpha$ in $\{0,\dots,d\}^{d+1}$, with $z \in B_\alpha$ if and only if ${(\lfloor (d+1)\gamma_1\rfloor,\dots,\lfloor (d+1)\gamma_{d+1}\rfloor) = \alpha}$, where the $\gamma_i$'s are the coefficients of the representation of $z$ in terms of the generators $(1,v_i)$.

\begin{proposition}\label{heightProp} Let $z$ be an integer point in $B_\alpha$. Then
\[
\height(z)=\left\lceil\frac{\sum_{i=1}^{d+1}\alpha_i}{d+1}\right\rceil
.\]
\end{proposition}
\begin{proof}
Considering the zeroth coordinate of $z$, it is clear that $\height(z)=\sum_{i=1}^{d+1}\gamma_i$.

Note that since ${\alpha=(\lfloor (d+1)\gamma_1\rfloor,\dots,\lfloor (d+1)\gamma_{d+1}\rfloor) }$,
 the inequality \[(d+1)\gamma_i-1<\lfloor (d+1)\gamma_i\rfloor\leq (d+1)\gamma_i\]
  implies that \[\sum_{i=1}^{d+1}\big((d+1)\gamma_i-1\big) \;<\; \sum_{i=1}^{d+1}\alpha_i\;\leq \sum_{i=1}^{d+1} (d+1)\gamma_i.\] 
  
 Thus 
\[
(d+1)\big(\height(z)-1\big)<\;\sum_{i=1}^{d+1}\alpha_i\;\leq (d+1)\height(z),
\]and 
\[\height(z)-1<\frac{\sum_{i=1}^{d+1}\alpha_i}{d+1}\leq\height(z),
\]
from which the result follows.\end{proof}

This leads to the following characterization of the integer decomposition property:
\begin{corollary} \label{HBIDP.cor}The simplex $\Delta$ is IDP if and only if for each $z$ in the Hilbert basis of $\cone(\Delta)$, $z\in B_\alpha$ implies that $\sum_{i=1}^{d+1}\alpha_i$ is at most $d+1$.  
\end{corollary}

\subsection{The function $\IDP$}\label{IDPdefinition}
\begin{definition} $\IDP$ is the $0/1$ function which takes as input the vertices of a lattice $d$-simplex $\Delta$ and returns one if $\Delta$ is IDP and zero otherwise.
\end{definition}
 We were unsuccessful in approximating $\IDP$ directly using techniques presented in this paper. Instead, we find success approximating another function, $\HB$, from which the value of $\IDP$ can be inferred. 

For a vector $x\in\R^{(d+1)^{d+1}}$, we consider $x$ to be multi-indexed by the collection of $\alpha\in \{0,\dots,d\}^{d+1}$. 

\begin{definition}$\HB$ is the function taking as input the vertices of a lattice $d$-simplex and returning an element of  $\{0,1\}^{(d+1)^{d+1}}$, with coordinate $\alpha$ equal to one if and only if there exists a Hilbert basis element in bin $B_\alpha$.

\begin{example}\label{runningExample} Consider the $\Delta_{(1,q)}$ simplex in dimension $d=2$ with $q$-vector $(2,1)$. The ray generators are $v_1=(1,0,1)$, $v_2=(0,1,1)$, and $v_3=(-2,-1,1)$. Computation with {\bf{Normaliz}} \cite{Normaliz} yields that the set of lattice points in $\Pi_{\Delta_{(1,g)}}$ is equal to $\left\{(0,0,0),\;(-1,0,1),\;(0,0,1)\right\}$. The representation of these points in terms of the ray generators are $(0,0,0)$, $(0,\sfrac{1}{2},\sfrac{1}{2})$, and $(\sfrac{1}{2},\sfrac{1}{4},\sfrac{1}{4})$. The bins $B_\alpha$ which contain these points have $\alpha$ equal to: 
\begin{align*} 
&\left(\left\lfloor {3\cdot 0} \right\rfloor,\left\lfloor {3\cdot 0} \right\rfloor,\left\lfloor {3\cdot 0} \right\rfloor\right)=(0,0,0),\\
&\left(\left\lfloor {3\cdot 0} \right\rfloor,\left\lfloor {3\cdot \sfrac{1}{2}} \right\rfloor,\left\lfloor {3\cdot \sfrac{1}{2}} \right\rfloor\right)=(0,1,1), \text{ and }\\
&\left(\left\lfloor {3\cdot \sfrac{1}{2}} \right\rfloor,\left\lfloor {3\cdot \sfrac{1}{4}} \right\rfloor,\left\lfloor {3\cdot \sfrac{1}{4}} \right\rfloor\right)=(1,0,0),
\end{align*} respectively.

When the $\alpha$ are lexicographically ordered, we may write the image under $\HB$ as the vector \[(1, 0, 0, 0, 1, 0, 0, 0, 0, 1, 0, 0, 0, 0, 0, 0, 0, 0, 0, 0, 0, 0, 0, 0, 0, 0, 0)\in\R^{3^3}.\]
\end{example}

\end{definition}

Let $\supp$ be the $0/1$ function on $\R^{(d+1)^{d+1}}$ which, for a vector $x$, returns zero if and only if there exists an index $\alpha$ such that $x_\alpha\neq0$ and $\sum_{i=1}^{d+1}\alpha_i>d+1$. Then, using Corollary \ref{HBIDP.cor}, we may write the functional equality 
\[\IDP=\supp\circ\HB.
\]

Note that for Example \ref{runningExample}, the non-zero entries are at multi-indices $(0,0,0)$, $(0,1,1)$ and $(1,0,0)$, and that the sum of each individual multi-index is not more than 3. Thus the image of $\IDP$ is equal to 1, indicating that the example is IDP. We can verify this fact by noting that the height of each lement of the Hilbert basis 
\[\left\{v_1,v_2,v_3\right\}\;\bigcup\;\big\{(-1,0,1),(0,0,1)\big\}
\] is equal to 1. We remark that it is not true in general that the Hilbert basis elements are the non-zero lattice points of $\Pi_\Delta$.

We have developed a theoretical framework for approximating the integer decomposition property by approximating the real-valued function $\HB$. One difficulty in the implementation of this scheme is the fact that $\supp$ is not sensitive to how close to zero a value is. If the entry at some multi-index $\alpha$ in the approximation of $\HB$ is close to but not equal to zero, and $\sum_{i=1}^{d+1}\alpha_i>d+1$, then the image of $\supp$ will be 0, i.e., our approximation of $\IDP$ will almost always predict that an example is {\emph{not}} IDP. A standard solution to this issue is to first map our approximation into the open interval $(0,1)$, then choose a value $0\leq\eta\leq1$, then interpret values less than or equal to $\eta$ as 0 and greater than $\eta$ as 1. For the the first step, we use the Sigmoid function:   
\[\sigma(x) := (1+e^{-x})^{-1},\] mapping $\R$ one-to-one onto the open interval $(0,1)$. 
For some fixed $0\leq\eta\leq1$, define 
\[\cutoff(x)=\begin{cases} 0& \text{if $x\leq\eta$, and } \\
1& \text{otherwise.}
\end{cases}\]
The composition of $\cutoff$ and $\sigma$ allows us to turn any real valued function of one variable into a $0/1$ function, and by applying it coordinate-wise, we may turn any function $f\;:\;\R^u\longrightarrow\R^v$ into $\cutoff\circ\;\sigma\circ f\;:\;\R^u\longrightarrow \{0,1\}^v$. 
In particular, consider $\R^u$ to be the space parameterizing the vertex sets of lattice $d$-simplices: $\R^u=\R^d\times\cdots\times\R^d$ with $d+1$ factors $\R^d$ (not all points in the space give rise to full-dimensional simplices.) Further consider $\R^v$ to have basis multi-indexed by $\alpha\in\{0,\dots,d\}^{d+1}$.
Then for any map $f$ from $\R^u$ to $\R^v$, $\cutoff\;\circ\;\sigma\;\circ\;f$ may be considered as a map from lattice $d$-simplices  to 0/1 vectors indexed by bins $B_\alpha$.

Continuing Example \ref{runningExample}, consider the function $f\;:\R^{2\times3}\longrightarrow\R^{3^3}$ defined by \[f(\x)=\frac{1}{\Vert \x\Vert}\left(\sum_{i=1}^{27}(-1)^i\cdot e_i\right).\] Then $f(1,0,0,1,-2,-1)=\frac{1}{\sqrt{7}}\left(\sum_{i=1}^{27}(-1)^i\cdot e_i\right)$. We compute that $\sigma(1/\sqrt7)=0.593$, and that $\sigma(-1/\sqrt7)=0.407$. Thus if $0\leq\eta<0.407$, then $\cutoff\circ\;\sigma\circ f(1,0,0,1,-2,-1)$ is the all-ones vector, and if $0.593\leq\eta\leq1$ then $\cutoff\circ\;\sigma\circ f(1,0,0,1,-2,-1)$ is the zero vector. If $0.407\leq\eta<0.593$, then $\cutoff\circ\;\sigma\circ f(1,0,0,1,-2,-1)$ is the 0/1-vector $\sum_{i=1}^{27}\frac{\left(1+(-1)^i\right)}{2}\cdot e_i$.

The quality of the approximation depends heavily on the choice of value for $\eta$, as for the fixed function $f$, $\cutoff\circ\;\sigma\circ f$ can be correct on $11\%$, $89\%$, or $33\%$ of the entries of $\HB$, depending on the choice of $\eta$.

\begin{definition} Let $f$ be any function from $\R^{d(d+1)}$\; to\; $\R^{(d+1)^{d+1}}$. Then we call the (coordinate-wise) composite  function \[\widehat{\IDP}:=
\supp\;\circ\;\cutoff\;\circ\;\sigma\;\circ\;f
\] an approximation of the integer decomposition property.\end{definition}
Note that when $\sigma\circ f$ closely approximates $\HB$ coordinate-wise, $\widehat{\IDP}$ agrees with $\IDP$. For a given $f$, we will use the shorthand notation $\widehat{\HB}$ for the composite function $\cutoff\;\circ\;\sigma\;\circ\;f$.
\section{A general approximation method} 
In this section we describe piece-wise linear functions as compositions of affine transformations and a well-behaved piece-wise linear function $\rho$. We next describe the use of a loss function $L$ in quantifying the accuracy of an approximation $\widehat{f}$ of a function $f$. We then describe an algorithm called gradient descent, which deforms the piece-wise linear function $\widehat{f}$ in order to minimize the loss function $L$ with respect to the target function $f$. 
  
Let $f$ be any set map from $\R^u$ to $\R^v$. We will approximate $f$ by constructing a {\emph{random}} initial ``approximation" $\widehat{f}$, which we will deform until we have a sufficiently accurate approximation. 
 
For a positive integer $m$, fix $m$ positive integers $\ell_1$ through $\ell_{m}$, as well as a small real ${\epsilon>0}$. We will call this the collection of {\bf{hyper-parameters}}. Choose matrices ${W_{k}\in\R^{\ell_{k-1}\times\ell_k}}$ for ${1\leq k\leq m+1}$, where we set $\ell_0=u$ and $\ell_{m+1}=v$ (the dimensions of the domain and codomain of $f$.) Additionally, for each $k$,  choose  vectors $b_{k}\in \R^{\ell_{k}}$. The entries $\left(W_k\right)_{i,j}$ are called {\bf{weights}}, and the $(b_k)_i$ are called {\bf{biases}}. Generally, the initial values are randomized by an algorithm we will not discuss here. We will consider each such collection of {\bf{parameters}} to be a point \[p=\big(W_1,b_1,\dots,W_{m+1},b_{m+1}\big)
\] in the space of parameters 
$\R^{(\ell_0+1)\times\ell_1}\times\cdots\times\R^{(\ell_{m}+1)\times\ell_{m+1}}$. We define $\rho$ to be the function which returns the coordinate-wise maximum of 0 and the identity, i.e., $\rho(x_i)=\max(0,x_i)$. The map $\rho$ is an example of an {\bf{activation function}} and is called {\bf{ReLU}} (Rectified Linear Unit.)
Let $\omega_k$ be the affine map $x\mapsto W_k(x)+b_{k}$ composed with $\rho$.  Then the approximation $\widehat{f}$ is the function 
\[\widehat{f}(x;p) = W_{m+1} \circ \;\omega_{m}\circ\cdots\circ \;\omega_1(x)+b_{m+1}.\]

\begin{center}
\begin{figure}[h]
\begin{tikzpicture}
\node (input) {$\R^u$};
\node (L1) [Li , right of =input,xshift = .5cm] {$\omega_1$};
\node (L2) [Li , right of =L1,xshift = .5cm] {$\omega_2$};
\node (dots) [right of =L2,xshift = .25cm] {$\cdots$};
\node (Lk) [Li , right of =dots,xshift = .25cm] {$\omega_m$};
\node (Weights) [weights , right of =Lk,xshift=.5cm] {$W_{m+1}$ };
\node (Biases) [biases, right of =Weights,xshift=.5cm] {$b_{m+1}$};

\node (output) [right of = Biases,xshift = .5cm] {$\R^v$};

\draw [->] (input) -- (L1);
\draw [->] (L1) -- (L2);
\draw [-] (L2) -- (dots);
\draw [->] (dots) -- (Lk);
\draw [->] (Lk) -- (Weights);
\draw [->] (Weights) -- (Biases);
\draw [->] (Biases) -- (output);

\draw[rounded corners=5pt,thick]
  (3/4,-1) rectangle ++(8.5,2);
\end{tikzpicture}
\caption{The approximation $\widehat{f}$}
\end{figure}
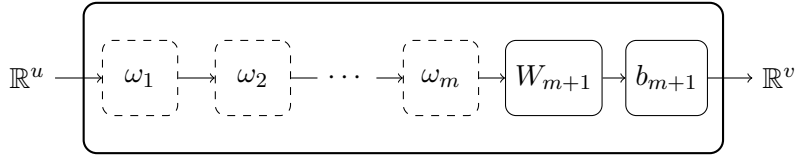
\end{center}

\begin{example}\label{logExample} Let $f(x)=\log(x)$. We will approximate $f$ on the interval $[1,3]$. Let $m=1$ and $\ell_1=2$. We initially set the parameters $p=(\omega_1,b_1,\omega_2,b_2)\in\R^{(1+1)\times 2}\times\R^{(2+1)\times1}$ by \[W_1=[0.75,-0.5]^T\quad b_{1}=[-0.75,1]\quad W_2=[1,1]\quad b_{2}=[-0.5].\] 

The resulting approximation, which we expect to be poor because it knows nothing about the function it is supposed to approximate, is given by the piece-wise linear function (the dotted graph in Figure~\ref{firstApprox.figure}) 
\begin{align}\label{equationForLater}\widehat{f}(x;p)&=[1,1]\rho\left(
\begin{bmatrix}
    0.75 \\
    -0.5 
\end{bmatrix}[x]+\begin{bmatrix}
    -0.75 \\
    1 
\end{bmatrix}\right)+[-0.5]\nonumber\\
&=1\cdot\rho(0.75x-0.75)+1\cdot\rho(-0.5x+1)-0.5\nonumber\\
&=\begin{cases} 
      0.25x-0.25 & 1\leq x\leq 2 \\
       0.75x-1.25& 2< x \leq 3
   \end{cases}\nonumber
\end{align}
\end{example}

\begin{center}
\begin{figure}[h]\label{firstApprox.figure}
\begin{tikzpicture}[scale=.75]
\begin{axis}
    \addplot[domain=1:3] {ln(x)};
    \addplot[domain=1:2,dashed,thick] {(0.25*x)-0.25};
    \addplot[domain=2:3,dashed,thick] {(0.75*x)-1.25};
    
\end{axis}
\end{tikzpicture}
\caption{The function $f$ and the approximation $\widehat{f}$ (dashed)}
\end{figure}
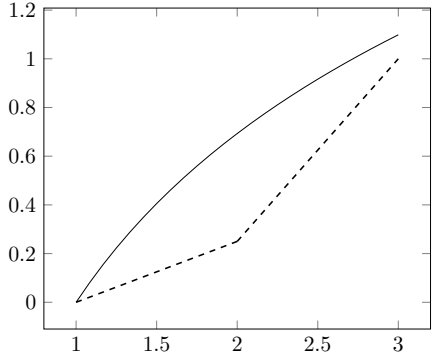
\end{center}

\subsection{Loss functions and gradient descent} 
 We measure the quality of the approximation via a {\bf{loss function}} $L(x;p)$ which we attempt to minimize. By minimizing its value at many ``training" points $x$ distributed throughout the domain, we hope that the value of the approximation $\widehat{f}$ will be close to that of $f$ at points outside of training set, i.e., that the magnitude of the loss function will be small at new points as well.
 
 One example of a loss function is the Euclidean distance 
 \[D(x;p) = \left\lVert f(x) - \widehat{f}(x,p) \right\rVert.\]
 Continuing Example \ref{logExample}, 
 \[D(x;p) = \left\lVert \log(x) - \big((W_2)_{1,1}\rho\big((W_1)_{1,1}(x)+b_{1,1}\big)+(W_2)_{1,2}\rho\big((W_1)_{2,1}(x)+b_{1,2}\big)+b_{2,1}\big) \right\rVert,\] and for our specific parameters $p$,
 \[D(x;p) = \left\lVert \log(x) - \big(1\cdot\rho\big(0.75x-0.75\big)+1\cdot\rho\big(0.75x+1\big)-0.5\big) \right\rVert.\]

Although for fixed parameters $p$, the loss $L(x;p)$ is a function of $x$, the ``learning" step of machine learning happens by interpreting it as a function of $p$, holding $x$ fixed. We can imagine $L$ as a surface above the parameterization space which is fixed by the choice of hyper-parameters and $x$. In order to improve our approximation $\widehat{f}$ at a particular point $x$ in the domain, we modify its parameters in such a way that that the value of the loss function $L$ is reduced, i.e., ``moving downhill" on the surface $L$. 

We compute the gradient $\nabla L$ with respect to the parameters $p$ at the point $(x,p)$ and update the parameters by $p\mapsto p - \epsilon \nabla L$. The value of $\epsilon$ is chosen small enough that $L(x;p - \epsilon \nabla L)<L(x;p)$. When we repeatedly apply this process for points $x$ sampled uniformly at random, this method is called {\bf{stochastic gradient descent}} or SGD. In practice, for reasons of computational efficiency and stability, a batch of points are sampled and the mean of the gradients is used for the update. This is known as {\bf{mini batch}} SGD.

Continuing our example, fix $x=1.5$ and use the chain rule to compute that
\begin{align*}\nabla D(1.5;p)&= \left\langle \frac{\partial D}{\partial\omega_1}\,,\,\frac{\partial D}{\partial b_1}\,,\,\frac{\partial D}{\partial\omega_2}\,,\,\frac{\partial D}{\partial b_2} \right\rangle_{x=1.5}\\
&= \left\langle  -1.5 \,,\, -1.5 \,,\, -1 \,,\, -1 \,,\, -0.375 \,,\, -0.25\,,\, -1 \right\rangle.\end{align*}

Then for $\epsilon=0.02$, the update $p' = p-\epsilon\,\nabla D(1.5;p)$ is given by 
\[\omega_1=[0.78\,,\,-0.47]^T\quad b_1=[-0.73\,,\,1.02]\quad\omega_2=[1.0075\,,\,1.0075]\quad\,b_2=[-0.48].\]
 The resulting updated approximation is \[\widehat{f}(x;p') = \begin{cases} 
      0.312x-0.187 & 1\leq x\leq 2.17 \\
       0.786x-1.215 & 2.17< x \leq 3
   \end{cases}.\]

\begin{center}
\begin{figure}[h]
\begin{tikzpicture}[scale=.75]
\begin{axis}
    \addplot[domain=1:3] {ln(x)};
    \addplot[domain=1:2,dotted,thick] {(0.25*x)-0.25};
    \addplot[domain=2:3,dotted,thick] {(0.75*x)-1.25};
    \addplot[domain=2.17:3,dashed, thick] {.786*x-1.215};
    \addplot[domain=1:2.17,dashed, thick] {0.312*x-.187};
    
\end{axis}
\end{tikzpicture}
\caption{The approximations $\widehat{f}(x;p)$ (dotted) and $\widehat{f}(x;p')$ (dashed)}
\end{figure}
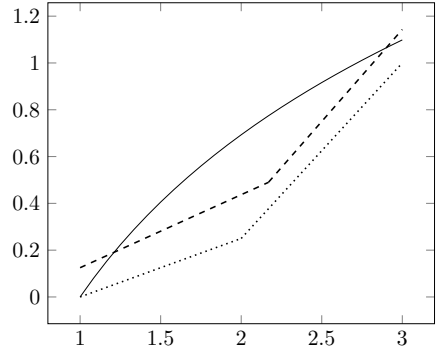
\end{center}

\begin{center}
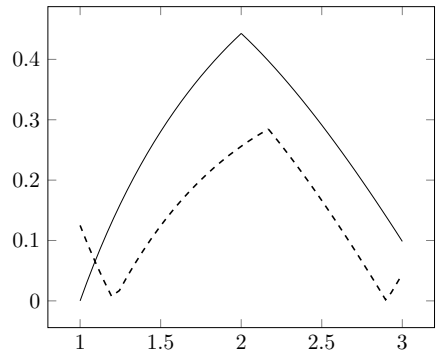
\begin{figure}[h]\label{firstLoss.figure}
\begin{tikzpicture}[scale=.75]
\begin{axis}
    \addplot[domain=1:2] {abs(ln(x)-(0.25*x-0.25))};
    \addplot[domain=2:3] {abs(ln(x)-.75*x+1.25)};
    \addplot[domain=2.17:3,dashed, thick] {abs(ln(x)-.786*x+1.215)};
    \addplot[domain=1:2.17,dashed, thick] {abs(ln(x)-0.312*x+.187)};
    
\end{axis}
\end{tikzpicture}
\caption{The loss function $D$ for $\widehat{f}(x;p)$ and $\widehat{f}(x;p')$ (dashed)}
\end{figure}
\end{center}

\subsection{Training and Validation} In practice, we perform the update step many thousands of times at $x$-values distributed throughout the domain. Often we gather a large collection of pairs $(x,f(x))$ called a {\bf{training set}} to store for later use in the update process, rather than computing the value of $f$ when needed. When generating this collection is costly, as in the case of the function $\IDP$, we use each pair from the collection multiple times over, in some cases as many as 100 times. By analogy with polynomial approximation, where we fit a polynomial to a finite set of points on the graph of a function, one may wonder if, when reusing sampled points in refining our approximation, we are simultaneously losing accuracy at other points in the domain. The short answer is yes. 

This phenomenon of {\bf{overfitting}} is a principal concern in the process of refining our approximation, and there are some standard techniques for mitigating its effect, including:
\begin{itemize}
\item creating two collections of pairs $(x,f(x))$ --- one for training and one for validation. As we train $\widehat{f}$ we simultaneously monitor its accuracy on the {\bf{validation set}}. If the performance on the validation set worsens while improving on the training set, we stop training.
\item introducing a component to the loss function for the magnitudes of the parameters. Experience shows that this method, called {\bf{regularization}}, reduces overfitting to the training data.
\item using the simplest ``structure" possible to achieve the desired performance. Complicated models require more training to achieve their optimal performance, and hence increase the number of times training data is reused. We balance the expressive capability of a complicated approximation with the need to minimize overfitting. 
\end{itemize}

\section{Implementation and Results} \label{ImplementationandResults}
\subsection{Implementation}  Our first goal is an approximation of the function $\HB$ restricted to the vertex sets of $\Delta_{(1,q)}$ simplices of dimension $d=4$ and with $q$-vector bounded by 25. Recall also that, even though the target space of $\HB$ has dimension $(d+1)^{d+1}$, the relevant values are those at indices $\alpha$ whose coordinates sum to more than $d+1$. We restrict to these 2,877 relevant indices. Hence the input to our function is the tuple $(q_1,q_2,q_3,q_4)\in[1,25]^4$ and the output is in $\R^{2,877}$.

There is no general-purpose best design of hyper-parameters that works for every application of a neural network. In fact, it is possible to approximate with arbitrary accuracy any continuous function on a compact subset of $\R^u$ using only one ``hidden layer" ($m=1$.) The general rule is that higher values of $m$ allow smaller values for the $\ell_i$'s while maintaining approximation flexibility. Optimizing hyper-parameters is a process that is outside the scope of this work, so we will simply report that, after experimenting with several values of $m$ and $\ell_i$'s in order to minimize the loss function and computation time, we proceeded using the following choice of hyper-parameters: 
\smallskip

\begin{center}
\begin{tabular}{|c|c|c|c|c|c|}
\hline
\rowcolor{gray!25}$m$&$\ell_1$&$\ell_2$&$\ell_3$&$\ell_4$&$\epsilon$\\\hline
4&100&400&800&3,000&0.001\\\hline

\end{tabular}
\end{center}
\medskip

Tensorflow produces a neural network with the specified dimensions and initializes the weights and biases automatically. In order to implement mini batch SGD, the user must make more decisions than just specifying the hyper-parameters.

\begin{enumerate}
\item Amount of training/validation data:\\
\quad we used Normaliz and a script to compute $\HB$ for 50,000 examples sampled uniformly at random, 10\% of which we reserved for validation.
\item Batch size:\\\quad during training we computed the gradient $\nabla L$ for batches of 10 $q$-vectors at a time and used the mean for the update of the parameters $p$. 
\item Loss function:\\\quad Because the image of $\HB$ is contained by the set $\left\{0,1\right\}^{2,877}$, we may consider the approximation to be the composite $\sigma\widehat{f}$ and use the {\bf{Binary Cross Entropy}} loss function $\BCE$ summed entry-wise over \[\BCE=\left(\HB-1\right)\cdot\log\,\left(1-\sigma\widehat{f}\right)-\HB\cdot\log\,\sigma\widehat{f}.\] When the value of $\HB$ is one, the  value of $\BCE$ is decreased by increasing the value of $\log\sigma\widehat{f}$, i.e., increasing the value of $\widehat{f}$. In this case, minimizing $\BCE$ coincides with minimizing the difference between $\sigma\widehat{f}$ and $\HB$. A similar analysis for the case when $\HB$ equals zero shows that $\BCE$ is a measure of the accuracy of $\sigma\widehat{f}$ as an approximation of $\HB$.

We used a modification of $\BCE$, which we discuss in Section \ref{whyBalance}
\item Training length:\\\quad We performed roughly 100,000 updates in the process of training the approximation.
\end{enumerate}

The result of this training procedure was a piece-wise linear function $f$. It was the well-defined and deterministic\footnote{For purposes of analysis and reproducibility, we initialize the computer's randomness generator so that the stochastic processes are, in fact, deterministic, while still having good randomness properties.} result of the specific choices outlined above. 

An approximation $\widehat{\HB}$ requires a choice of cutoff parameter $\eta$, and $\widehat{\IDP}$ requires the additional choice of a tolerance parameter $\tau$ (introduced in Subsection \ref{whyBalance}). These parameters control the functions $\cutoff$ and $\supp$, respectively. Recall that the resulting approximations are given by 
\[\widehat{\HB}:=\cutoff\;\circ\;\sigma\;\circ\;f\qquad \text{and } \qquad \widehat{\IDP}:=
\supp\;\circ\;\widehat{\HB}
\] We present the results in terms of the values $\eta$ and $\tau$.

\subsection{The approximation $\widehat{\HB}$}\label{whyBalance}
\
While it is tempting to present the accuracy of $\widehat{\HB}$ as the percentage of indices on which it agrees with $\HB$, this is problematic due to the scarcity of non-zero entries in any given image of $\HB$. Consider the $q$-vector $(4,10,14,14)$; there are just 14 non-zero entries among the 2,877 relevant entries in its image under $\HB$. Consequently, an approximation which is uniformly equal to zero would be correct 99.5\% of the time, while knowing essentially nothing about the function it is trying to approximate other than that it is typically equal to zero! We therefore present the accuracy in the form of a {\bf{confusion table}}, which breaks down the indices $\alpha$ along two criteria --- firstly depending on whether   $\HB_\alpha$ is equal to 1 (positive) or 0 (negative), and secondly whether $\widehat{\HB}_\alpha$ is positive or negative. 

\begin{example} Again using the $q$-vector $(4,10,14,14)$, we set $\eta=0.1$ and present the resulting confusion table below: 
\begin{center}
\begin{table}[h]
\begin{tabular}{|c|c|c|}
\hline
	&PREDICTED 0 & PREDICTED 1 \\
	 \hline
  ACTUAL 0&2,808 & 55 \\ 
  \hline
  ACTUAL 1&0 &14\\
   \hline
 \end{tabular}
  \medskip
 \caption{The confusion table for $\widehat{\HB}(4,10,14,14)$ }
 \end{table}
 \end{center}
 \vspace{-24pt}
 
Observe that the sum of the table entries is, in fact, 2,877. We call entries appearing in the upper right cell of the table ``{\bf{false positive}}"  because the approximation incorrectly predicted that a bin contained a Hilbert basis element. Similarly, entries in the bottom left cell are called  ``{\bf{false negative}}". 

We may summarize the table with the pair of ratios \begin{align*}\text{specificity}&=\frac{\text{true negatives}}{\text{true negatives}+\text{false positives}}  \\[1ex]\text{sensitivity}&=\frac{\text{true positives}} {\text{true positives} + \text{false negatives}}.\end{align*} For the present example, they are $98\%$ and $100\%$, respectively. The specificity and sensitivity vary with the cutoff value $\eta$, and are negatively correlated with each other, as demonstrated in Table \ref{varyingAccuracyEta}:
\begin{center}
\begin{table}[h]
\begin{tabular}{|c|c|c|}
	\hline
	$\eta$&specificity& sensitivity\\\hline
	0.1&0.981& 1.00 \\\hline
  0.25&0.986 & 0.857\\\hline
    0.5&0.993 & 0.214\\\hline
 \end{tabular}
 \medskip
 \caption{The effect of varying $\eta$}
 \label{varyingAccuracyEta}
 \end{table}
 \end{center}
\end{example}
\vspace{-24pt}
\subsection{Validation} When we sampled 50,000 examples for training, we reserved 5,000 of them for validation purposes. We now report the performance on this validation set, which we denote $\cS$.
 We aggregate (sum entry-wise) the confusion tables for $\eta=0.1$ in Table \ref{aggregatedTable}.

\begin{center}
\begin{table}[h]
\begin{tabular}{c|c|c|}
	&PREDICTED 0 & PREDICTED 1 \\
	 \hline
  ACTUAL 0&12,726,675 & 1,573,167 \\ 
  \hline
  ACTUAL 1&22,569 & 88,482\\
   \hline
 \end{tabular}
  \medskip
 \caption{An aggregated confusion table for $\cS$ ($\eta=0.1$)}
 \label{aggregatedTable}
 \end{table}
 \end{center}
 
The corresponding aggregated specificity is $89.0\%$, and sensitivity is $79.7\%$. One can account for the difference between specificity and sensitivity by recalling the scarcity of non-zero entries of $\HB$, i.e., the low total number of positives. If we use the loss function 
 \[\BCE=\left(\HB-1\right)\cdot\log\,\left(1-\sigma\widehat{f}\right)-\HB\cdot\log\,\sigma\widehat{f}\] as earlier described for our gradient descent, the resulting approximation will essentially be the constant zero function. In order for the model to learn to identify positives, we must balance the contributions to the loss function associated to positive and negative according to the inverse of their frequency. We accomplish this by introducing a positive term $\beta$ which we call the {\bf{balance}} term:
\[L=\left(\HB-1\right)\cdot\log\,\left(1-\sigma\widehat{f}\right)-{\bf{\beta}}\cdot\HB\cdot\log\,\sigma\widehat{f}.\] The results presented in this section correspond to a $\beta$ value of 10. All other parameters remaining fixed, a higher value, roughly $\beta=75$, is required in order achieve approximately equal sensitivity and specificity. However, it is not necessarily desirable to match the sensitivity and specificity, as we will discuss.

\subsection{The approximation $\widehat{\IDP}$} Under the unrealistic assumption that Hilbert basis elements are distributed roughly uniformly among bins, consider an approximation with a specificity of 99.9\% applied to the $q$-vector of an IDP $\Delta_{(1,q)}$ simplex. Because there are 2,877 bins, the probability that all bins will be correctly identified as  negative (not containing a Hilbert basis element) can be estimated as  $0.999^{2,877}\approx 5.6\%$. Since we expect the incidence of the integer decomposition property to be low, a true positive rate for IDP of $5.6\%$ may result in few or even no examples being correctly predicted as IDP! We have several tools to combat this issue:
\begin{enumerate}
\item manipulating the balance term $\beta$ to produce high specificity (possibly at the expense of sensitivity)
\item manipulating the cutoff value $\eta$ to produce high specificity (again, at the expense of sensitivity)
\item tolerating some number of  positive entries in $\widehat{\HB}$ (under the assumption that many of them are false.)
\end{enumerate}

For this last option we introduce the {\bf{tolerance}} parameter $\tau$, which sets an upper bound on the number of positive entries before the function $\widehat{\IDP}$ returns that an example is IDP negative. In our original description of $\widehat{\IDP}$, $\tau$ was implicitly set to zero.

Table \ref{specificityTable} records the number of true positives over the total number of positives of $\widehat{\IDP}$ when applied to the sample $\cS$ for select values of $\eta$ and $\tau$. 
\begin{center}
\begin{table}[h]
\begin{tabular}{|c|c|c|c|c|}\hline\normalfont
\diaghead(1,-1){\hspace{24pt}}%
{$\tau$}{$\eta$}&
0.5&0.25&0.12&0.05\\    \hline
0& 3/7 (42.9\%)& 3/4 (75.0\%) &3/3 (100.0\%)&3/3 (100.0\%)\\    \hline
10& 21/320 (6.6\%)&  11/38 (29.0\%)&8/21 (38.1\%)&6/12 (50.0\%)\\    \hline
20& 46/1026 (4.5\%)& 21/102 (20.6\%)&11/45 (24.4\%)&8/27 (29.6\%)\\    \hline
30& 65/1770 (3.7\%)& 35/196 (17.9\%)&23/103 (22.3\%)&16/64 (25.0\%)\\    \hline
\end{tabular}
\medskip
\caption{The rate of true positives (specificity) for given values of $\eta$ and $\tau$}
\label{specificityTable}
\end{table}
\end{center}
From Table \ref{specificityTable}, we see that there is not one optimal choice for the values of $\eta$ and $\tau$, since higher specificity is correlated with few examples being found; the goals of specificity and sensitivity are in tension. Figure \ref{tradeoff} shows the (log-scale) relationship between specificity and sensitivity induced by varying these values. When we actually checked for IDP using Normaliz, we found 112 positive examples among 5,000. The analogous ``specificity" is 2.24\%, but the ``sensitivity" is 100\% ---  we plot this point $(2.24,100)$ for reference. 

\begin{center}
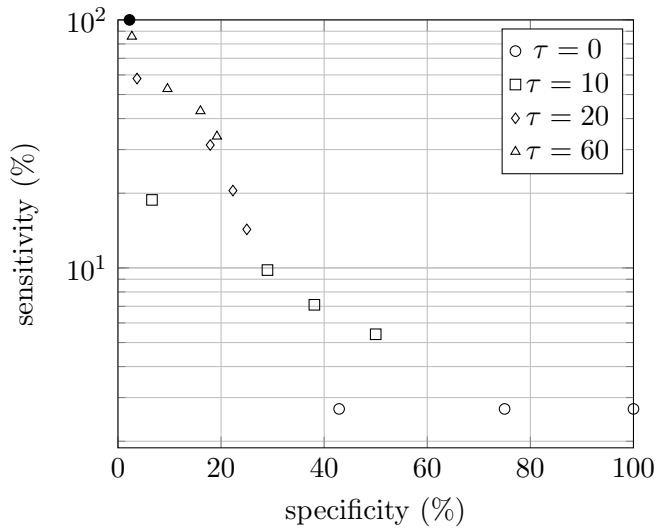
\begin{figure}[h]
\begin{tikzpicture}
\begin{axis}[xmin=0, ymin=0, xmax=100,ymax=100,xlabel={specificity ($\%$)},
ylabel={sensitivity ($\%$)}, ymode= log, grid = both,
scatter/classes={
a={mark=o,draw=black},b={mark=square,draw=black},d={mark=diamond,draw=black},e={mark=triangle,draw=black},f={mark=*,draw=black}}]
\addplot[scatter,only marks,%
    scatter src=explicit symbolic]%
table[meta=label] {
x y label
42.9 2.7 a
75 2.7 a
100 2.7 a
6.6 18.8 b
29 9.8 b
38.1 7.1 b
50 5.4 b
3.7 58 d
17.9 31.3 d
22.3 20.5 d
25 14.3 d
2.7 85.7 e
9.6 52.7 e
16 42.9 e
19.2 33.9 e
2.24 100 f
    };
\addplot +[mark=none,color=black,dashed] coordinates {(2.24,0) (2.24, 100)};
\legend{$\tau=0$,$\tau=10$,$\tau=20$,$\tau=60$}
\end{axis}
\end{tikzpicture}
\caption{The effect of varying $\eta$ for fixed values of $\tau$}
\label{tradeoff}
\end{figure}
\end{center}

 Table~\ref{IDPexamples} lists all 112 $q$-vectors of $\cS$ that correspond to IDP $\Delta_{(1,q)}$ simplices according to Normaliz. Recall that the rate of IDP in $\cS$ is 2.24\%. Table~\ref{predicted00} lists the subset of $\cS$ which are predicted to be IDP when $\eta=0.5$ and $\tau=0$, with the correct positive predictions highlighted. Observe that the incidence of IDP among the {\emph{predicted}} IDP examples is about 43\%, much higher than the rate in the sample at large. 
 
 We highlight the $q$-vectors in Table~\ref{IDPexamples} that correspond to true IDP positive predictions made by setting $\eta=0.1$ and $\tau=65$ (the specificity was 15\% and the sensitivity was 58\%.)

  \begin{center}
 \begin{table}[h]
 \setlength{\tabcolsep}{16pt}
\begin{tabular}{ccccc}
\cellcolor{gray!25}1,1,1,1&\cellcolor{gray!25}1,1,3,9&\cellcolor{gray!25}1,1,21,24&1,2,14,10&1,2,14,10\\
\cellcolor{gray!25}1,3,16,3&\cellcolor{gray!25}1,3,24,1&1,4,2,16&1,4,20,20&\cellcolor{gray!25}1,8,1,1\\
\cellcolor{gray!25}1,10,10,8&1,10,24,24&1,12,4,12&\cellcolor{gray!25}1,15,3,1&\cellcolor{gray!25}1,18,1,6\\
\cellcolor{gray!25}1,21,1,4&\cellcolor{gray!25}1,24,1,9&1,24,14,2&\cellcolor{gray!25}1,24,17,1&\cellcolor{gray!25}1,24,18,1\\
1,24,18,4&\cellcolor{gray!25}1,24,24,20&\cellcolor{gray!25}2,2,2,7&2,3,12,18&2,8,8,4\\
2,10,1,16&2,20,10,5&\cellcolor{gray!25}3,1,1,9&3,6,12,1&3,12,2,24\\
3,14,21,3&3,19,3,1&3,23,15,3&\cellcolor{gray!25}4,1,1,4&4,8,2,16\\
4,20,1,14&4,20,10,20&4,23,4,12&4,24,1,16&6,1,2,12\\
6,2,6,3&6,2,18,9&\cellcolor{gray!25}6,6,6,3&6,14,6,15&6,17,9,18\\
7,3,21,7&\cellcolor{gray!25}7,7,1,7&7,7,16,16&8,1,8,2&8,2,12,24\\
8,16,4,2&\cellcolor{gray!25}9,1,1,9&9,6,18,2&9,9,4,4&9,18,4,4\\
9,18,18,6&9,22,1,11&10,1,5,22&10,5,10,9&10,24,4,1\\
11,22,5,5&\cellcolor{gray!25}12,1,2,6&12,1,24,19&\cellcolor{gray!25}12,2,3,12&12,2,18,3\\
12,3,2,6&12,3,11,6&\cellcolor{gray!25}12,6,1,1&12,6,1,3&\cellcolor{gray!25}12,12,4,12\\
12,16,1,16&12,24,2,24&12,24,6,1&13,2,2,20&14,6,14,7\\
14,7,2,24&14,7,12,1&\cellcolor{gray!25}15,1,13,15&\cellcolor{gray!25}15,15,1,1&16,1,6,6\\
\cellcolor{gray!25}16,4,2,16&16,7,16,16&16,8,4,2&16,16,12,3&16,24,1,22\\
\cellcolor{gray!25}17,1,7,1&17,17,8,4&17,17,17,1&\cellcolor{gray!25}18,1,1,15&18,2,6,6\\
18,2,22,1&18,10,1,15&19,19,1,16&20,2,1,12&20,8,19,8\\
20,14,24,1&\cellcolor{gray!25}20,20,1,20&20,20,4,1&20,20,4,20&20,22,1,22\\
21,21,16,4&\cellcolor{gray!25}22,2,2,22&22,16,4,1&22,16,22,1&22,22,20,1\\
23,2,2,6&23,18,3,24&23,24,24,12&24,2,1,16&24,4,2,4\\
24,24,6,24&24,24,23,12&&&
\end{tabular}
\medskip
\caption{The 112 IDP examples in the sample $\cS$}
\label{IDPexamples}
\end{table}

 \begin{table}[h]
 \setlength{\tabcolsep}{16pt}
\begin{tabular}{ccccc}
\cellcolor{gray!25}1,1,1,1&1,2,10,2&\cellcolor{gray!25}1,3,24,1&\cellcolor{gray!25}2,2,2,7&2,3,4,7\\
4,3,2,5&11,6,9,6&&&
\end{tabular}
\medskip
\caption{Predicted IDP examples ($\eta=0.5$, $\tau=0$)}
\label{predicted00}
\end{table}
\end{center}

\subsection{Discussion} As a demonstration of the utility of the approximation method presented here, we could attempt to advance the previously mentioned work in \cite{2016arXiv160801614B} on $\Delta_{(1,q)}$ simplices by producing a large and diverse collection of IDP examples from which to form conjectures to try to prove. A natural scheme for arriving at such a collection is to first generate a test set, say, all $\Delta_{(1,q)}$ simplices of dimension $d$ with $q$-vector entries bounded by $n$, then verify the integer decomposition property with a program like Normaliz, collecting the positive examples. We could augment this scheme with machine learning by performing an initial sieving step prior to testing with Normaliz. By developing a computationally--cheap approximation to the integer decomposition property, we can reserve the relatively expensive Normaliz computations for those examples that, according to the approximation, are more likely to be IDP. 

In the context of this application, the results outlined above point to a tradeoff between the computational efficiency (controlled by the specificity) and the number of examples that are ultimately produced (controlled by the sensitivity). It also seems that the approximation $\widehat{\IDP}$ is biased in favor of repeated entries (see the highlighted examples in Table \ref{IDPexamples},) which brings into question how diverse a set of examples it is capable of producing.

We computed the value of $\widehat{\IDP}$ for all 390,625 $\Delta_{(1,q)}$ simplices with $q$-vector in $[1,25]^4$ using $\eta=0.1$ and $\tau=65$. The computation took 3,3562 seconds and produced 2,520 predicted positives. We then computed $\IDP$ for these examples and found that 521 were IDP. This corresponds to a specificity of 20.7\%. It is impractical to compute $\IDP$ over the entire collection of 390,625 examples in order to compute the sensitivity, so it is not known. 


\section{Concluding Remarks} It is very likely that other choices of hyper-parameters, or even entirely different machine learning techniques, will yield improved performance. However, the results, such as they are, do indicate that functions like $\IDP$ have the potential to be modeled by machine learning techniques. The following remarks point out directions in which this investigation might be continued.

\begin{remark}Figure \ref{tradeoff} shows the tradeoff between specificity and sensitivity for an approximation $\widehat{\IDP}$ that is a product of a choice of hyper-parameters, balance $\beta$, and training size. It would be useful to see the effect of different values of $\beta$ in the plot. Does there exist a choice which achieves sensitivity and specificity of 50\%?
\end{remark}

\begin{remark} The intermediate step of computing an approximation of $\HB$ has several potential applications which are not explicitly discussed in this paper. In particular we note that by computing the set of lattice points in each predicted--positive bin, we have an approximation of the Hilbert basis itself. 

If the sensitivity of $\widehat{\HB}$ is high then it is very likely that the Hilbert basis is {\emph{contained}} by the approximated Hilbert basis, and may be recovered by the reduction algorithm used by Normaliz (implemented by a python script, for example.) This could potentially be more efficient than Normaliz, which reduces the entire fundamental parallelepiped, if the specificity is high.
\end{remark}

\begin{remark}The Ehrhart $h^*$-vector records the number of lattice points at each height in $\Pi_\Delta$. In the case that the $h^*$-vector is the concatenation of two vectors --- the first increasing and the second decreasing --- we call it {\bf{unimodal}}. Unimodality is another interesting property to investigate for lattice simplices, see, e.g., \cite{unimodality}. If, rather than recording the presence of a Hilbert basis element, we were to record the number of fundamental parallelepiped points in each bin, we could approximate the $h^*$-vector using Proposition \ref{heightProp}. Thus we have a framework for predicting both IDP {\emph{and}} unimodality.
\end{remark}

\section*{Acknowledgement}The author thanks Devin Willmott and Kyle Helfrich for many helpful conversations.
\providecommand{\bysame}{\leavevmode\hbox to3em{\hrulefill}\thinspace}
\providecommand{\MR}{\relax\ifhmode\unskip\space\fi MR }
\providecommand{\MRhref}[2]{%
  \href{http://www.ams.org/mathscinet-getitem?mr=#1}{#2}
}
\providecommand{\href}[2]{#2}


\begin{thebibliography}{1}

\bibitem{2016arXiv160801614B}
B.~{Braun}, R.~{Davis}, and L.~{Solus}, \emph{{Detecting the Integer
  Decomposition Property and Ehrhart Unimodality in Reflexive Simplices}},
  ArXiv e-prints (2016).

\bibitem{unimodality}
Benjamin Braun, \emph{Unimodality problems in {E}hrhart theory}, Recent trends
  in combinatorics, IMA Vol. Math. Appl., vol. 159, Springer, [Cham], 2016,
  pp.~687--711. \MR{3526428}

\bibitem{Normaliz}
W.~Bruns, B.~Ichim, T.~R\"€omer, R.~Sieg, and C.~S\"oger, \emph{Normaliz.
  algorithms for rational cones and affine monoids}, Available at
  \text{https://www.normaliz.uni-osnabrueck.de}.

\bibitem{DeepLearningBook}
Ian Goodfellow, Yoshua Bengio, and Aaron Courville, \emph{Deep learning}, MIT
  Press, 2016, \url{http://www.deeplearningbook.org}.

\bibitem{conjugacy}
Jonathan Gryak, Robert Haralick, and Delaram Kahrobaei, \emph{Solving the
  conjugacy decision problem via machine learning},  (2017).

\bibitem{travelingSalesman}
Jean-Yves Potvin, \emph{State-of-the-art survey?the traveling salesman problem:
  A neural network perspective}, ORSA Journal on Computing \textbf{5} (1993),
  no.~4, 328--348.

\bibitem{understandability}
Geoffrey~G Towell and Jude~W Shavlik, \emph{Extracting refined rules from
  knowledge-based neural networks}, Machine learning \textbf{13} (1993), no.~1,
  71--101.

\end{thebibliography}
\end{document}